\documentclass{article}

\usepackage[utf8]{inputenc}
\usepackage[margin=0.7in]{geometry}
\usepackage{colonequals, multicol, multirow, booktabs}

\usepackage[usenames,dvipsnames]{xcolor} 
\definecolor{darkgreen}{rgb}{0,0.5,0}
\definecolor{darkblue}{rgb}{0,0,0.8}
\definecolor{darkred}{rgb}{0.8,0,0}

\usepackage[pdfencoding=auto,colorlinks,citecolor=darkgreen,linkcolor=darkblue,urlcolor=darkred]{hyperref}
\usepackage{float}
\usepackage[
backend=biber,
style=alphabetic,
sorting=nyt,
backref=true,
maxitems=10,
minitems=10,
maxbibnames=10,
minbibnames=10,
maxcitenames=10,
mincitenames=10,
giveninits=true,
maxalphanames=7,
minalphanames=7,
doi=false,
isbn=false,
url=false
]{biblatex}
\usepackage{amsmath, verbatim, amsthm, tikz}
\usepackage{amssymb}
\usepackage{authblk}

\newtheorem{theorem}{Theorem}[section]
\newtheorem{lemma}[theorem]{Lemma}
\newtheorem{proposition}[theorem]{Proposition}
\newtheorem{definition}[theorem]{Definition}

\newcommand{\Z}{\mathbf{Z}}
\newcommand{\Q}{\mathbf{Q}}
\newcommand{\Qp}{\mathbf{Q}_p}
\newcommand{\Fp}{\mathbf{F}_p}

\renewcommand{\H}{\mathrm{H}}

\DeclareMathOperator{\Ker}{Ker}
\DeclareMathOperator{\NS}{NS}
\DeclareMathOperator{\Aut}{Aut}

\title{Quadratic Chabauty for Atkin-Lehner Quotients of Modular Curves of Prime Level and Genus 4, 5, 6}
\author[1]{Nikola Adžaga}
\affil[1]{Department of Mathematics, Faculty of Civil Engineering, University of Zagreb\thanks{nadzaga@grad.hr }}

\author[2]{Vishal Arul}
\affil[2]{Department of Mathematics, University College London\thanks{varul.math@gmail.com}}

\author[3]{Lea Beneish}
\affil[3]{Department of Mathematics and Statistics, McGill University\thanks{lea.beneish@mail.mcgill.ca}}

\author[4]{Mingjie Chen}
\affil[4]{Department of Mathematics, University of California at San Diego\thanks{mic181@ucsd.edu}}

\author[5]{Shiva Chidambaram}
\affil[5]{Department of Mathematics, University of Chicago\thanks{shivac@uchicago.edu}}

\author[6]{Timo Keller}
\affil[6]{Department of Mathematics, Universit\"{a}t Bayreuth\thanks{Timo.Keller@uni-bayreuth.de}}

\author[7]{Boya Wen}
\affil[7]{Department of Mathematics, Princeton University\thanks{boyaw@math.princeton.edu}}

\date{\today}

\begin{document}

\maketitle
\begin{abstract}
    We use the method of quadratic Chabauty on the quotients $X_0^+(N)$ of modular curves $X_0(N)$ by their Fricke involutions to provably compute all the rational points of these curves for prime levels $N$ of genus four, five, and six. We find that the only such curves with exceptional rational points are of levels $137$ and $311$. In particular there are no exceptional rational points on those curves of genus five and six. More precisely, we determine the rational points on the curves $X_0^+(N)$ for $N=137,173,199,251,311,157,181,227,263,163,197,211,223,269,271,359$.

    \emph{Keywords:} Rational points; Curves of arbitrary genus or genus $\neq 1$ over global fields; arithmetic aspects of modular and Shimura varieties
    
    \emph{2020 Mathematics subject classification:} 14G05 (primary); 11G30; 11G18
\end{abstract}

\section{Introduction}

{The curve $X_{0}^{+}(N)$ is a quotient of the modular curve $X_{0}(N)$ by the Atkin-Lehner involution $w_{N}$ (also called the Fricke involution).  The non-cuspidal points of $X_{0}^{+}(N)$ classify unordered pairs of elliptic curves together with a cyclic isogeny of degree $N$ between them, where the Atkin-Lehner involution $w_{N}$ sends an isogeny to its dual. The set $X_{0}^{+}(N)(\mathbf{Q})$ consists of cusps, points corresponding to CM elliptic curves (CM points), and points corresponding to quadratic $\mathbf{Q}$-curves without complex multiplication. The points of the third kind are referred to as \emph{exceptional points}.

There have been several works related to the study of $\mathbf{Q}$-rational points on Atkin-Lehner quotients of modular curves (see \cite{BDMTV, Castano-Bernard, DL, EL, Galbraith_1996,Galbraith_1999,Galbraith_2002, Mercuri_2018, Momose_1987}) and on Atkin-Lehner quotients of Shimura curves (see \cite{Clark, PY}). Especially relevant for this work are the articles \cite{Galbraith_1996,Galbraith_1999,Galbraith_2002} in which Galbraith constructs models of all such curves of genus $\leq 5$ except for $X_0^+(263)$, and conjectures that he has found all exceptional points on these curves. Building on work of Galbraith, Mercuri \cite{Mercuri_2018} constructs models for such curves of genus $6$ and $7$ of prime level and shows that up to a (very) large naive height, there are no exceptional points on six of these curves (those with $N=163, 197, 229, 269$ and $359$). 

 Towards verifying Galbraith’s conjecture, Balakrishnan, Dogra, M\"uller, Tuitman, and Vonk \cite{BDMTV} show that of the curves $X_{0}^{+}(N)$ with prime $N$ of genus $2$ and $3$, the only curves with exceptional rational points are those with level $N=73,103,191$. The rational points on $X_0^+(N)$ for $N = \{67,73,103\}$ were computed by Balakrishnan, Best, Bianchi, Lawrence, M\"uller, Triantafillou, and Vonk \cite{BBBLMTV} using quadratic Chabauty and the Mordell-Weil sieve. The remaining genus $2$ levels, $\{107,167,191 \}$, are more challenging because there are not enough rational points to run the algorithm. Balakrishnan, Dogra, M\"uller, Tuitman, and Vonk \cite{BDMTV} compute the rational points on the remaining curves of genus $2$ and prime level as well as those of genus $3$ and prime level using models computed by Galbraith \cite{Galbraith_1996} and Elkies.}

{Quadratic Chabauty refers to the technique of depth-two Chabauty-Kim. The Chabauty--Kim method~\cite{Kim_2005,Kim_2009} can be seen as a generalization of the usual Chabauty--Coleman method. The usual Chabauty--Coleman method is a useful tool in computing rational points on curves (see \cite{Chabauty, Coleman, McPoonen, Stoll}), however there are still large classes of curves that do not satisfy the rank condition $r := \operatorname{rk}\operatorname{Jac}(X)(\mathbf{Q}) < g$ of Chabauty--Coleman. Kim developed a nonabelian version of the Chabauty-Coleman method, in which one replaces the use of the Jacobian of a curve with a Selmer variety associated to a certain Galois-stable unipotent quotient of the $\mathbf{Q}_p$-pro-unipotent completion of the étale fundamental group of the curve. Despite the technical nature of Kim's setup, the work of Balakrishnan-Dogra and Balakrishnan, Dogra, M\"uller, Tuitman, and Vonk \cite{BD1, BD2, BDeffective, Balakrishnan_2019}, have shown that one can practically implement a depth-two version of Kim's program.  In particular, work of Siksek \cite{Siksek} uses the criterion of Balakrishnan-Dogra for finiteness of the quadratic Chabauty set to show that for modular curves of genus $g\geq 3$, quadratic Chabauty is more likely to succeed than classical Chabauty--Coleman. Further,  Balakrishnan,  Best,  Bianchi,  Dogra,  Lawrence, M\"uller, Triantafillou, Tuitman, and Vonk developed computational tools to carry out quadratic Chabauty explicitly (see \cite{balakrishnan2020computational,BDMTV,BBBDLMTTV}).
}

Continuing the work begun in \cite{BDMTV}, we aim to compute all $\mathbf{Q}$-rational points on the curves $X_0^+(N)$ of prime level $N$ when the genus is $4, 5, 6$ using the quadratic Chabauty method \cite{BD1, BD2} and the computational tools developed in \cite{BBBDLMTTV}. From Section \ref{sec:genus}, we have that for prime $N$, the curve $X_{0}^{+}(N)$ has genus 4 if and only if \begin{equation}\label{genus4}
    N \in \{ 137, 173, 199, 251, 311 \}.
\end{equation} 
It has genus 5 if and only if \begin{equation}\label{genus5}
    N \in \{ 157, 181, 227, 263 \},
\end{equation}
and it has genus 6 if and only if \begin{equation}\label{genus6}
    N \in \{ 163, 197, 211, 223, 269, 271, 359 \}.
\end{equation}
These modular curves are not hyperelliptic according to \cite{HasegawaHashimoto}. They satisfy the conditions required for quadratic Chabauty, because their Mordell-Weil rank is equal to their genus and their Picard number is greater than $1$. In addition, they usually have enough rational points (according to~\cite[Remark~1.6]{Balakrishnan_2019}, $g+1$ points suffice when using equivariant $p$-adic heights, and one can lower this requirement by constructing $g$ independent rational points on the Jacobian of the curve).

{To use the quadratic Chabauty method, we use models for $X_{0}^{+}(N)$ computed by Galbraith 
\cite{Galbraith_1996,Galbraith_1999,Galbraith_2002} and \cite{Mercuri_2018}. 
Equations for $N \in \{ 137, 157 \footnote{The canonical model for $X_{0}^{+}(157)$ given in Section \ref{sec:Models}
corrects a small typo in the third equation in Galbraith's model for $N = 157$ (the leading term should be $2 w^{2}$ instead of $w^{2}$).}, 173, 181, 199, 227, 251  \}$ are taken from \cite[pp.~32--34]{Galbraith_1996}, the equation for $N = 311$ is taken from
\cite[pp.~316]{Galbraith_1999}, and equations for  $N \in \{ 163, 197, 211, 223, 269, 271, 359\}$ are taken from \cite[pp.~299--306]{Mercuri_2018}. The rational points are taken from \cite[pp.~88]{Galbraith_1996}, \cite[pp.~316]{Galbraith_1999}, and \cite[pp.~299--306]{Mercuri_2018}. For $N \in \{ 199, 251 \}$, we computed the known rational points by brute force. These equations and the list of known rational points in each of these levels are given in Section~\ref{sec:Models}. With these models, we find suitable plane affine patches (in the sense of the last paragraph of Section~\ref{sec:QC}) for the code \texttt{QCModAffine} developed in~\cite{BBBDLMTTV} and verify that the list of rational points is complete.  

Our main result is the following:

\begin{theorem} The only genus $4$ curves $X_0^+(N)$ {where $N$ is prime} that have exceptional rational points are $X_0^+(137)$ and $X_0^+(311)$. There are no exceptional rational points on curves $X_0^+(N)$ of genus $5$ and $6$ {for any prime $N$.} 
\end{theorem}}
Therefore, this work confirms Galbraith's conjecture for prime levels.
In particular, this, combined with the results of \cite{BDMTV} on genus $2$ and $3$ curves, combined with works of \cite{AraiMomose_2010,BBBLMTV,Balakrishnan_2019, BarsGonzalezXarles, Momose_1986, Momose_1987}, and work in progress of the second author and M\"uller on curves $X_{0}^{+}(N)$ with composite level will resolve Galbraith's conjecture. For genus $6$, this work shows that the subset of rational points that Mercuri found is in fact all of them, and for those curves with $N=163,197, 359$ we confirm \cite[Result 1.2]{Mercuri_2018} without any restrictions on height. 

{This paper is organized as follows. In Section \ref{sec:QC}, we give an overview of the quadratic Chabauty method. Section \ref{sec:strategy} contains a description of the method we used, namely how we found the models that we used to run the quadratic Chabauty algorithm starting from the canonical model of $X_{0}^{+}(N)$ and how we chose suitable primes. In Section \ref{sec:genus}, we bound the genus of $X_{0}^{+}(N)$ from below and give a list of all $N$ such that the genus of $X_{0}^{+}(N)$ is less than or equal to $6$. In Section \ref{sec:Models}, there are tables for genus $4,5,6$ containing the canonical models of $X_{0}^{+}(N)$, the known rational points, and our models and primes for which we run quadratic Chabauty.}

 \subsection*{Acknowledgments}
This project was initiated as part of a $2020$ Arizona Winter School project led by Jennifer Balakrishnan and Netan Dogra. The authors are grateful to Jennifer Balakrishnan and Netan Dogra for the project idea and for their valuable guidance throughout the process. The authors are also grateful to Steffen M\"uller, Padmavathi  Srinivasan,
and Floris Vermeulen for their support during the project and to Jennifer Balakrishnan, Netan Dogra, and Steffen M\"uller for their comments on an earlier draft.
N.\@ A.\@ was supported by the Croatian Science Foundation under the project no. IP2018-01-1313.

\section{Overview of Quadratic Chabauty for Modular Curves: theory and algorithm}\label{sec:QC}

Let $X/\Q$ be a smooth projective geometrically connected curve of genus $g \geqslant 2$ with Jacobian $J$ whose Mordell-Weil group $J(\Q)$ has rank $r = g$. For a prime $p$ of good reduction, the abelian logarithm induces a homomorphism
\[ \log : J(\Q_p) \rightarrow \H^0(X_{\Qp}, \Omega^1)^{\vee}.\]
We assume that the $p$-adic closure of the image of $J(\Q) \subset J(\Qp)$ in $\H^0(X_{\Qp}, \Omega^1)^{\vee}$ has rank $g$ (otherwise classical Chabauty-Coleman applies). We also assume that $X(\Q)$ is non-empty, so we can choose a rational base point $b$ to map $X$ into its Jacobian $J$.
If the Néron-Severi rank $\rho$ of $J$ is larger than $1$, then there exists a non-trivial $Z \in \Ker(\NS(J)  \rightarrow \NS(X))$ inducing a correspondence on $X \times_\Q X$. Balakrishnan and Dogra in \cite{BD1,Balakrishnan_2019}
explain how to attach to any such $Z$ a locally analytic \emph{quadratic Chabauty function} \[\rho_Z : X(\Qp)\rightarrow \Qp \] as follows: using Nekovář's theory of $p$-adic heights \cite{Nekovar},
one can construct a global $p$-adic height which decomposes as a sum of local height functions. The quadratic Chabauty function $\rho_Z$ is defined as the difference between the global $p$-adic height and the local height for the chosen prime $p$. Even though we do not go into details here, we note that the computation of the global height pairing is easier when $X$ has many rational points (at least $3$, which is indeed the case for all modular curves studied in this paper).

The crucial property of the quadratic Chabauty function is that there exists a finite set $\Upsilon\subset \Qp$, such that $\rho_Z(x)\in\Upsilon$ for any $x\in X(\Q)$. Because $\rho_Z$ has Zariski-dense image on every residue disk and is given by a convergent power series, this implies that the set $X(\Q_p)_2$ of $\Q_p$-points of $X$ having values in $\Upsilon$ under $\rho_Z$ is finite. Because $X(\Q)$ is contained in that set, it is finite as well. Since both $\rho_Z$ and $\Upsilon$ can be explicitly computed by \cite{BBBDLMTTV}, this makes the provable determination of $X(\Q_p)_2$ possible\footnote{at least under the assumption that $\rho_Z$ has no repeated roots; otherwise, we get a finite superset.} if the quadratic Chabauty condition $r < g + \rho - 1$ is satisfied (e.g., if $r = g$ and $\rho > 1$).

In \cite{BDMTV}, the authors applied quadratic Chabauty to some of the modular curves associated to congruence subgroups of $\mathrm{SL}_2(\mathbf{Z})$, as well as Atkin-Lehner quotients of such curves. The algorithm in \cite{BDMTV} is specific to modular curves only when determining the nontrivial class $Z$. In their algorithm, $Z$ is computed by the Hecke operator $T_p$ (which is determined by the Eichler-Shimura relation). We recall the inputs and outputs of the algorithm in \cite{BDMTV} here.

\begin{itemize}
\item Input:
    \begin{itemize}
        \item {A modular curve $X/\Q$ that satisfies $r=g \geq 2$ and such that the $p$-adic closure of the image of $J(\Q)$ in $\H^0(X_{\Qp}, \Omega^1)^{\vee}$ under $\log$ has rank $g$.}
        \item{ A prime $p$ which is a prime of good reduction for $X/\Q$ such that the Hecke operator $T_p$ generates $\operatorname{End}(J) \otimes_\Z \Q$.} We check this condition in our cases by showing that $a_p(f)$ generates the number field generated by the coefficients of $f$ where $f$ is any newform orbit representative associated to $X/\Q$. 
    \end{itemize}
\item Output: A finite set containing $X(\Q_p)_2$. 
\end{itemize}

The curves of our interest, $X_0^+(N)$ of genus 4, 5 and 6 where $N$ is a prime, satisfy the condition $r=g$. This is done by checking that for all $N$ in~\eqref{genus4}, \eqref{genus5} and~\eqref{genus6}, $f\in S_2(\Gamma_0(N))^{+,\mathrm{new}}$ satisfies $L'(f,1) \neq 0$ (actually, we just need to check this for one arbitrary representative in each Galois orbit). According to~\cite[Proposition 7.1]{DL} (which the authors attribute to Kolyvagin and Logachev~\cite{KolyvaginLogachev}), this implies that the (algebraic and analytic) rank of $J_0^+(N)/\Q$ equals $\sum \dim(A_{f_i}) = g$ where the summation is taken over an arbitrary set of newform orbits representatives. Let us note here that this equality $r=g$ doesn't hold for all $X_0^+(N)$, the smallest genus where $r>g$ happens (for prime $N$) is $g=206$ (and $N=5077$) \cite{BDMTV,DL}.

The implementation (available at \cite{BBBDLMTTV}) of the algorithm in \cite{BDMTV} is designed to take as input a plane affine patch $Y: Q(x,y)=0$ of a modular curve $X/\Q$. The model $Q(x,y)=0$ does not need to be smooth, but it should be monic in $y$ with $p$-integral coefficients. We can sometimes find an affine patch $Y$ such that all rational points on $X$ must be among the points returned by running their program on $Y$. We have managed to do that for all genus $4$ and $5$ curves $X_0^+(N)$ where $N$ is prime.
If no such $Y$ is found, then we find two affine patches such that every rational point on $X$ is contained in at least one patch. Moreover, we need every $\Fp$-point on $X$ (realized as a canonical model) to map to a good point (Definition \ref{def1}) on at least one patch. Finding favorable affine patches turned out to be the most challenging part when applying their algorithm for large genus Atkin-Lehner quotients. We talk about our strategy for generating these affine patches in 
Section~\ref{sec:strategy}.

\section{Overview of strategy}\label{sec:strategy}

As introduced in Section~\ref{sec:QC}, our first task is to find a suitable plane model $Q(x,y)$ {of $X_0^+(N)$} that is monic in $y$ and has ``small'' coefficients. We start with the image of the canonical embedding of $X_0^+(N)$ in $\mathbf{P}^{g-1}$ (as stated in the introduction, these are non-hyperelliptic for our $N$). In what follows we shall identify $X_0^+(N)$ with its canonical image and use the same notation for both. We find two rational maps $\tau_{x}, \tau_{y} \colon X_{0}^{+}(N) \to \mathbf{P}^{1}$ such that the product $\tau_{x} \times \tau_{y} \colon X_{0}^{+}(N) \to \mathbf{P}^{1} \times \mathbf{P}^{1}$ is a birational map onto its image. Compose with the Segre embedding $\mathbf{P}^{1} \times \mathbf{P}^{1} \to \mathbf{P}^{3}_{[w: x: y: z]}$, and then project from the point $[1 \colon 0 \colon 0 \colon 0]$ onto the plane $w = 0$; 
denote by $\varphi'$ the composite 
\[
\varphi' \colon X_{0}^{+}(N) \xrightarrow{\tau_{x} \times \tau_{y}} \mathbf{P}^{1} \times \mathbf{P}^{1} \xrightarrow{\text{Segre}} \mathbf{P}^{3}_{[w: x: y: z]} \xrightarrow{\text{projection}} \mathbf{P}^{2}_{[x: y: z]}.
\]
Choose $x_{1}, x_{2}, y_{1}, y_{2} \in \mathcal{O}_{\mathbf{P}^{g - 1}}(1)$ such that 
$\tau_{x}(q) = [x_{1}(q): x_{2}(q)]$ and $\tau_{y}(q) = [y_{1}(q): y_{2}(q)]$ in coordinates. 
Then the equation for $\varphi'$ is 
\[
\varphi' \colon q \mapsto [(x_{1} y_{2})(q) \colon (x_{2} y_{1})(q) \colon (x_{2} y_{2})(q)].
\]
(When $x_{2}(q), y_{2}(q) \neq 0$, this is just  $\varphi' \colon q \mapsto [(x_{1} / x_{2})(q): (y_{1} /
y_{2})(q): 1]$.) 

{Note that on some closed subvarieties of $X_0^+(N)$, the map $\varphi '$ is not defined. Indeed, on the subvariety with additional equations $x_1=x_2=0$, the map $\tau_x$ is not defined; with $y_1=y_2=0$, the map $\tau_y$ is not defined; with $x_2=y_2=0$, (the map $\tau_x$ or $\tau_y$ or) the projection following the Segre embedding is not defined. We denote the union of the three subvarieties by $X_0^+(N)_{\varphi' \text{undef}}$, outside which $\varphi'$ is well-defined. The set of rational points $X_0^+(N)_{\varphi' \text{undef}}(\Q)$ can be computed directly using Magma~\cite{Magma}, and in our final step, we check that it does not contain any additional points other than the known rational points. For now, we focus on the generic part of $X_0^+(N)$ where $\varphi'$ is defined.} 

The image $\mathcal{C}'_{N} \colonequals \varphi'(X_{0}^{+}(N))$ will be a curve 
given by an equation of the form
\[
    Q_{0}(x, z) y^{d} + Q_{1}(x, z) y^{d - 1} + \cdots + Q_{d}(x, z) = 0 ,
\]
where $Q_{0}(x, 1)$ is monic. (If not, simply divide the equation throughout by the leading coefficient of $Q_{0}(x, 1)$ to get to this form.) Multiply the above equation throughout by $Q_0^{d-1}$, we get
\[
    (Q_{0}(x, z) y)^{d} + Q_{1}(x, z)(Q_{0}(x, z) y )^{d - 1} + \cdots + Q_{0}(x, z)^{d-1} Q_{d}(x, z) = 0 ,
\]
Let $\deg Q_0$ denote the total degree of $Q_0$ and let $\mathcal{C}_{N}$ be the image of $\mathcal{C}'_{N}$ under the map 
\[
\psi \colon [x \colon y \colon z] \mapsto \left[ x z^{\deg Q_{0}} \colon Q_{0}(x, z) y \colon z^{\deg Q_{0} + 1}  \right].
\]
Then the affine patch of $\mathcal{C}_{N}$ given by $z = 1$ will have an equation $Q(x,y)=0$ where $Q(x,y)$ is a polynomial over $\Q$ monic in $y$, suitable for the quadratic Chabauty algorithm. 
Let $\varphi \colon X_{0}^{+}(N) \to \mathcal{C}_{N}$ denote the composite $\varphi = \psi \circ \varphi'$. 

Next, we select a suitable prime $p$ and run \texttt{QCModAffine} on the affine patch of $\mathcal{C}_{N}$ 
given by $z = 1$. However, \texttt{QCModAffine} only computes rational points (and compares with known rational points) \emph{outside} the \emph{bad} residue disks.
\begin{definition}[\cite{BalTui}, Definition~2.8 and~2.10]\label{def1}
Let $X^\mathrm{an}$ denote the rigid analytic space over $\mathbf{Q}_p$ associated to a projective curve $X/\mathbf{Q}_p$ given by (inhomogeneous) equation $Q(x,y)=0$. Let $\Delta(x)$ be the discriminant of $Q$ as a function of $y$, and let $r(x)=\Delta/\gcd(\Delta,d\Delta/dx)$. We say that a residue disk (as well as any point inside it) is \emph{infinite} if it contains a point whose $x$-coordinate is $\infty$, is \emph{bad} if it contains a point whose $x$-coordinate is {$\infty$ or} a zero of $r(x)$, and is \emph{good} if it is not \emph{bad}.
\end{definition}

Therefore, we may need to repeat this construction multiple times (with the same $p$), i.e., we 
need a collection $\{ (\varphi_{N,i}, \mathcal{C}_{N, i}) \}_{i =  1}^{k}$ such that for each $P \in 
X_{0}^{+}(N)(\mathbf{F}_{p})$, there exists some $i$ such that $\varphi_{N,i}(P)$ does not map to a bad $\mathbf{F}_{p}$-point. This would imply that for each residue disk $D$ of 
$X_{0}^{+}(N)(\mathbf{Q}_{p})$, there exists some $i$ such that 
$\varphi_{N,i}(D) \subseteq \mathcal{C}_{N, i}(\mathbf{Q}_{p})$ is contained in a good residue disk. If \texttt{QCModAffine} reports that the only rational points in the good
disks of $\mathcal{C}_{N, i}(\mathbf{Q}_{p})$ are the images of the known rational points, then we know that 
$X_{0}^{+}(N)(\mathbf{Q})$ is contained in the finite set 
\[
\bigcup_{i = 1}^{k} \varphi_{N, i}^{-1} \{ \varphi_{N, i}( X_{0}^{+}(N)(\mathbf{Q})_{\text{known}} ) \} \cup X_0^+(N)_{\varphi' \text{undef}}(\Q),
\]
and then we check that this equals $X_{0}^{+}(N)(\mathbf{Q})_{\text{known}}$ {using direct Magma computations}. 

In practice, we seek to choose $x_{1}, x_{2}, y_{1}, y_{2} \in 
\mathcal{O}_{\mathbf{P}^{g - 1}}(1)$ so that $d_{x}:=\deg\tau_x, d_{y}:=\deg \tau_y$ are small for faster computations, since the degree of the defining equation $Q(x,y)$ of $\mathcal{C}_{N}$ is at most $d_{x} d_{y}$ and the maximum power of 
$y$ that appears is $d_{x}$. We would also like to keep $k$ small, and ideally to have $k=1$, in which case we need to find a prime $p$ for which $Q(x,y)=0$ has no bad disks. In order to do so, we would like to start with $Q(x,y)$ for which $r(x)$ has no linear factor over $\mathbf{Q}$. If the construction does not satisfy this property, we adjust the construction of $\varphi'$ by applying some $\sigma\in \Aut(\mathbf{P}^1\times_\Q \mathbf{P}^1)$ before the Segre embedding, and/or by post-composing $\varphi '$ by some $\rho \in \Aut(\mathbf{P}^2)$ (the latter only invoked in the $N=211$ case). With suitable adjustments in the construction of $\varphi'$, we were able to find plane models $Q(x,y)$ for which $r(x)$ has no linear factor over $\mathbf{Q}$ for all the cases we aim to solve in this paper. Even so, for four of the genus~$6$ cases, $N=197, 211, 223, 359$, we did not find a model for which there is a prime $p$ without any bad disks, and in each of these four cases, we were able to find two patches and a prime $p$ such that for each $P \in X_{0}^{+}(N)(\mathbf{F}_{p})$,
$\varphi_{N,i}(P)$ does not map to a bad $\mathbf{F}_{p}$-point on at least one of the patches. For all the other prime levels where $X_0^+(N)$ has genus $4,5,6$, we were able to find a single patch and a prime for which there is no bad disk. That is, we were able to solve the problem with $k\le 2$.

In this section, we do not use that the curve in question has a modular interpretation. The strategy could possibly be applied to other curves of arbitrary genus $g$. Currently Magma supports computing plane models and gonal maps for curves up to genus $6$.

\section{Modular curves \texorpdfstring{$X_{0}^{+}(N)$}{X\_\{0\}\^{}\{+\}(N)} of genus at most 6}\label{sec:genus}
We will obtain an explicit lower bound on the genus 
of $X_{0}^{+}(N)$ in order to find all $X_{0}^{+}(N)$ of genus at most $6$.

Let $g_{0}(N)$ be the genus of $X_{0}(N)$ and let $g_{0}^{+}(N)$ be the genus of $X_{0}^{+}(N)$.

\begin{theorem}
\label{Theorem:LowerBoundg0N}
For integers $N \geq 1$, we have $g_{0}(N) \ge (N - 5 \sqrt{N} - 8) / 12$.
\end{theorem}

\begin{proof}
See Section 3 of \cite{Csirik_Wetherell_Zieve_2000}.
\end{proof}

Let $\nu(N)$ denote the number of fixed points of $w_{N}$. Let $h(D)$ denote the class number of the quadratic order with discriminant $D$.
\begin{theorem}
\label{Theorem:nuNComputation}
For $N \ge 5$,
\[
\nu(N) = \begin{cases}
h(-4 N) + h(-N) & \text{if } N \equiv 3 \pmod{4} \\
h(-4 N) &\text{otherwise.}
\end{cases}
\]
\end{theorem}
\begin{proof}
Apply the formula in Remark 2 of \cite{Furumoto_Hasegawa_1999}.
\end{proof}

To bound $g_{0}^{+}(N)$, we will need an upper bound on $h(D)$.

\begin{lemma}
\label{Lemma:UpperBoundOnClassNumber}
For negative $D$, 
\[
h(D) \le \frac{\sqrt{-D}}{\pi} (\ln(4 |D|) + 2).
\]
\end{lemma}
\begin{proof}
For negative $D$, the Dirichlet class number formula states that 
\[
h(D) = \frac{w \sqrt{-D}}{2 \pi} L\left(1, \chi_{D}\right)
\]
where $\chi_{D}(m) = \left( \frac{D}{m} \right)$ is the Kronecker symbol and
\[
w = \begin{cases}
4 &\text{if } D = -4, \\
6 &\text{if } D = -3, \\
2 &\text{otherwise.}
\end{cases}
\]
Hence, it suffices to obtain an upper bound on $L\left(1, \chi_{D}\right)$. Let 
$A(x) \colonequals \sum_{1 \le n \le x} \chi_{D}(n)$. 
Since the modulus of $\chi_{D}$ divides $4 D$, the sum of any $4 |D|$ consecutive values of 
$\chi_{D}(n)$ is zero, so
\[
|A(x)| \le 4 |D|.
\]
Using the definition of $L\left(1, \chi_{D}\right)$ and the Abel summation formula,
\begin{align*}
L\left( 1, \chi_{D} \right) &= \sum_{n = 1}^{\infty} \frac{\chi_{D}(n)}{n} \\
&= \sum_{n = 1}^{4 D} \frac{\chi_{D}(n)}{n} + \sum_{n > 4 D} \frac{\chi_{D}(n)}{n} \\
&= \sum_{n = 1}^{4 D} \frac{\chi_{D}(n)}{n} - \frac{A(4 D)}{4 D} - \int_{4 D}^{\infty} \frac{A(t)}{t^{2}} \, dt.
\end{align*} (Note that the series for $L(1,\chi_D)$ converges because $\chi_D$ is nontrivial.)
Since $A(4 D) = 0$, the middle term vanishes. Using the upper bounds $|\chi_{D}(n)| \le 1$ and $A(x) \le 4 |D|$
yields
\begin{align*}
L\left( 1, \chi_{D} \right) &\le \sum_{n = 1}^{4 |D|} \frac{1}{n} + 4 |D| \int_{4 |D|}^{\infty} \frac{1}{t^{2}} \, d t \\
&= \sum_{n = 1}^{4 |D|} \frac{1}{n} + 1 \\
&\le \ln(4 |D|) + 2
\end{align*}
from the Abel summation formula again. Substituting into the Dirichlet class number formula shows that for $D 
\not \in \{ -3, -4 \}$,
\[
h(D) \le \frac{\sqrt{-D}}{\pi} (\ln(4 |D|) + 2).
\]
We then manually check that this holds for $D \in \{ -3, -4 \}$ as well.
\end{proof}

\begin{proposition}
\label{Proposition:LowerBoundg0Np}
We have 
\[
g_{0}^{+}(N) \ge \frac{N - 5 \sqrt{N} + 4}{24} - \frac{\sqrt{N}}{\pi}(\ln(16 N) + 2).
\]
\end{proposition}
\begin{proof}
From the Riemann-Hurwitz formula applied to the degree-$2$ morphism $X_{0}(N) \to X_{0}^{+}(N)$, we have
\begin{equation}\label{eq:g0N and g0+N}
    2 g_{0}(N) - 2 = 2 (2 g_{0}^{+}(N) - 2) + \nu(N),
\end{equation}
and we are done by combining  
Theorem \ref{Theorem:LowerBoundg0N},
Theorem \ref{Theorem:nuNComputation}, and
Lemma \ref{Lemma:UpperBoundOnClassNumber}.
\end{proof}

\begin{proposition}
The complete list of $N$ for which $g_{0}^{+}(N) \le 6$ is as follows. 

\begin{table}[h!]
    \centering
    \begin{tabular}{cl}
        \toprule
        $g_{0}(N)^{+}$ & \multicolumn{1}{c}{$N$} \\
        \midrule
        $0$ & $2, 3, 5, 7, 11, 13, 17, 19, 23, 29, 31, 41, 47, 59, 71$ \\
        $1$ & $37, 43, 53, 61, 79, 83, 89, 101, 131$\\
        $2$ & $67, 73, 103, 107, 167, 191$\\
        $3$ & $97, 109, 113, 127, 139, 149, 151, 179, 239$\\
        $4$ & $137, 173, 199, 251, 311$\\
        $5$ & $157, 181, 227, 263$\\
        $6$ & $163, 197, 211, 223, 269, 271, 359$\\
        \bottomrule
    \end{tabular}
    \caption{The prime levels $N$ such that $X_0^+(N)$ has genus $\leq 6$}
    \label{tab:prime_levels}
\end{table}

\begin{table}[H]
    \centering
    \begin{tabular}{cl}
        \toprule
        $g_{0}(N)^{+}$ & \multicolumn{1}{c}{$N$} \\
        \midrule
        $0$ & $4, 6, 8, 9, 10, 12, 14, 15, 16, 18, 20, 21, 24, 25, 26, 27, 32, 35, 36, 39, 
        49, 50$ \\
        $1$ & $22, 28, 30, 33, 34, 38, 40, 44, 45, 48, 51, 54, 55, 56, 63, 64, 65, 75, 81, 
        95, 119$\\
        $2$ & $42, 46, 52, 57, 62, 68, 69, 72, 74, 77, 80, 87, 91, 98, 111, 121, 125, 143$\\
        $3$ & $58, 60, 66, 76, 85, 86, 96, 99, 100, 104, 128, 169$\\
        $4$ & $70, 82, 84, 88, 90, 92, 93, 94, 108, 115, 116, 117, 129, 135, 147, 155, 159, 
        161, 215$\\
        $5$ & $78, 105, 106, 110, 112, 122, 123, 133, 134, 144, 145, 146, 171, 175, 185, 209$\\
        $6$ & $118, 124, 136, 141, 152, 153, 164, 183, 203, 221, 299$\\
        \bottomrule
    \end{tabular}
    \caption{The composite levels $N$ such that $X_0^+(N)$ has genus $\leq 6$}
    \label{tab:composite_levels}
\end{table}
\end{proposition}
\begin{proof}
The lower bound in Proposition \ref{Proposition:LowerBoundg0Np} exceeds $6$ for $N > 13300$. 
For $N \le 13300$, we compute $g_{0}^{+}(N)$ exactly from the computation of $g_0(N)$ in~\cite[§\,3.1]{DiamondShurman}, Theorem~\ref{Theorem:nuNComputation} (the class numbers of imaginary quadratic fields can be computed efficiently using Magma, e.g.\ via the identification with the number of equivalence classes of reduced integral binary quadratic forms) and~\eqref{eq:g0N and g0+N} and find the $N$ for which $g_{0}^{+}(N) \le 6$.
\end{proof}

\section{Results and data}\label{sec:Models}

In the following tables, we display our choice of models for $X_{0}^{+}(N)$ for which we applied
\texttt{QCModAffine} successfully. We use $D$ when the rational point represents a
Heegner point of discriminant $D$. To compute the coordinates of a CM point of discriminant $D$, we compute the Heegner forms of elliptic curves with CM by an order of discriminant $D$, compute the corresponding $\tau$ in the upper half plane and $q = \exp(2\pi i \tau)$ and plug it into the $q$-expansions of the basis of the cusp forms giving the canonical embedding. For certain values of $D$ or $N$, the series did not converge, but if there is only one such $D$ for a given $N$, all $\Q$-points on $X_0^+(N)$ are known, and since we know their modular interpretation for all but one of them, we can also determine the interpretation of the remaining one. If all cusp forms vanish at the CM point corresponding to $D$ (this happens for $D = -3, -4$ in the case of $X_0^+(157)$ and $X_0^+(181)$, and $D = -3, -163$ in the case of $X_0^+(163)$), we compute the derivatives of $q$-expansions of the basis of the cusp forms giving the canonical embedding and divide by the value of one. In this way, we get the coordinates of the CM points with $D = -4$ and $-163$.

Use $d_{\infty}$ to denote the degree of the residue field at infinity. In the first table, we list $p$, $d_\infty$, $d_x$, $d_y$ and the runtime for all affine patches used.
The computations were performed on a six core server with $64$\,GB of RAM and processor Intel(R) Xeon(R) W-2133 CPU @ 3.60\,GHz.

\begin{center}
\begin{tabular}{c|ccccc}
$N$ & $p$ & $d_{\infty}$ & $d_{x}$ & $d_{y}$ & Runtime \\\hline
137 &  5 &  6 & 3 & 3 & $00\text{m}\ 36\text{s}\ 370\text{ms}$ \\
173 &  5 &  2 & 4 & 4 & $01\text{m}\ 58\text{s}\ 450\text{ms}$ \\
199 &  7 &  4 & 4 & 4 & $02\text{m}\ 23\text{s}\ 060\text{ms}$ \\
251 & 11 &  4 & 4 & 4 & $04\text{m}\ 39\text{s}\ 420\text{ms}$ \\
311 &  5 &  2 & 4 & 4 & $01\text{m}\ 53\text{s}\ 110\text{ms}$ \\ \hline\hline
157 &  5 &  4 & 4 & 5 & $06\text{m}\ 12\text{s}\ 150\text{ms}$ \\
181 &  7 &  6 & 3 & 5 & $03\text{m}\ 50\text{s}\ 490\text{ms}$ \\
227 & 23 &  6 & 3 & 5 & $03\text{m}\ 19\text{s}\ 710\text{ms}$ \\
263 & 23 &  4 & 6 & 6 & $22\text{m}\ 47\text{s}\ 600\text{ms}$ \\ \hline\hline
163 & 31 &  4 & 4 & 4 & $15\text{m}\ 59\text{s}\ 450\text{ms}$ \\
269 & 29 & 4 & 4 & 6 & $41\text{m}\ 23\text{s}\ 000\text{ms}$ \\ 
271 & 13 & 12 & 5 & 9 & $1235\text{m}\ 42\text{s}\ 060\text{ms}$ \\\hline
$197_1$ & 23 & 2 & 6 & 6 & $24\text{m}\ 09\text{s}\ 610\text{ms}$ \\
$197_2$ & 23 & 8 & 5 & 6 & $18\text{m}\ 48\text{s}\ 930\text{ms}$ \\ \hline
$211_1$ & 31 & 4 & 6 & 6 & $58\text{m}\ 42\text{s}\ 680\text{ms}$ \\
$211_2$ & 31 & 2 & 6 & 6 & $34\text{m}\ 52\text{s}\ 620\text{ms}$ \\ \hline
$223_1$ & 19 & 12 & 6 & 5 & $26\text{m}\ 04\text{s}\ 680\text{ms}$ \\
$223_2$ & 19 & 2 & 6 & 5 & $18\text{m}\ 56\text{s}\ 530\text{ms}$ \\ \hline
$359_1$ & 7 & 4 & 6 & 6 & $9\text{m}\ 54\text{s}\ 270\text{ms}$ \\
$359_2$ & 7 & 4 & 6 & 5 & $15\text{m}\ 41\text{s}\ 880\text{ms}$ \\ \hline
\end{tabular}
\end{center}
For $N=197,211,223,359$, runtimes are given for the first and the second patch listed in the tables later in this section, respectively. 

The models we have used are now listed and organized by the genus. We also write down the classification of all rational points.

\subsection{Genus 4}

\begin{center}
\begin{tabular}{|l|l|}
  \hline
  \multirow{34}{*}{$X_{0}^{+}(137)$} & \multirow{3}{14cm}{$g = 4$, equations for canonical embedding in $\mathbf{P}^{g - 1}$ are \begin{align*}
    & X Y + W Y + 2 Y^2 + 2 W Z + X Z + 6 Y Z + 3 Z^2 = 0, \\
    & X^3 + W X^2 + 6 X^2 Z - 2 X Y^2 - 5 X Y Z + X Z W + 13 X Z^2 + 2 Y^3 \\
    &\qquad+ 3 W Y^2 + W^2 Y + 3 W Y Z - 6 Y Z^2 + Z W^2 - 4 Z^2 W + 14 Z^3 = 0,
    \end{align*}}  \\ &  \\  &  \\  &  \\  & \\& \\
   &\multirow{5}{14cm}{with the variables corresponding to the following cusp forms \begin{align*}
    W &= -q + q^{2} - q^{3} + q^{4} + 3 q^{5} + 2 q^{6} + 4 q^{7} - 3 q^{8} + 2 q^{9} - 5 q^{10} - q^{11} - 2 q^{12} + 4 q^{13} - 5 q^{14} \\&\qquad+ 3 q^{15} + 2 q^{16} - 3 q^{17} - 2 q^{18} - q^{19} + O(q^{20}), \\
    X &= -q^{2} + 2 q^{3} + 2 q^{4} - 3 q^{5} - 2 q^{6} - 3 q^{7} + q^{8} + q^{9} + 4 q^{10} + 10 q^{11} - 2 q^{12} - q^{13} + 6 q^{14} \\&\qquad- 4 q^{15} - 7 q^{16} - 5 q^{17} + q^{18} - 7 q^{19} + O(q^{20}), \\
    Y &= 2 q^{3} - q^{4} - 3 q^{5} - 3 q^{6} - q^{7} + 3 q^{8} - 2 q^{9} + 6 q^{10} + 3 q^{11} + 3 q^{12} - q^{13} + q^{14} - 3 q^{16} \\&\qquad+ 2 q^{17} + 3 q^{18} + q^{19} + O(q^{20}), \\
    Z &= -q^{3} + 2 q^{5} + 2 q^{6} + q^{7} - q^{8} - 3 q^{10} - 4 q^{11} - q^{12} - q^{14} + 3 q^{16} + q^{17} - 2 q^{18} + 2 q^{19} + O(q^{20})
   \end{align*}} \\ & \\
\hline
   \end{tabular}
 \end{center}
According to Galbraith \cite{Galbraith_1999}, the $j$-invariant of the $\mathbf{Q}$-curve corresponding to the exceptional point is
\begin{align*} j = &\, (-423554849102365349285527612080396097711989843\\
&\quad \pm 9281040308790916967443095886224534005155665\sqrt{-31159})/2^{138}.
\end{align*}

 \begin{center}
  \begin{tabular}{|l|l|}
  \hline
   \multirow{38}{*}{$X_{0}^{+}(173)$} & \multirow{3}{14cm}{$g = 4$, equations for canonical embedding in $\mathbf{P}^{g - 1}$ are \begin{align*}
    &X^2 + W Y + X Y - W Z + 4 X Z - 3 Y Z = 0, \\
    &X Y^2 + Y^3 + W^2 Z + W X Z + 2 X^2 Z + 5 W Y Z + 6 X Y Z \\
    &\qquad+ 9 Y^2 Z + 3 W Z^2 + 11 X Z^2 + 14 Y Z^2 + 12 Z^3 = 0,
    \end{align*}}  \\ &  \\  &  \\  &  \\  & \\& \\
   &\multirow{5}{14cm}{with the variables corresponding to the following cusp forms \begin{align*}
    W &= q - 2 q^{4} - 2 q^{5} - 3 q^{7} - 4 q^{9} + 2 q^{11} + q^{12} + 3 q^{13} - q^{14} + q^{15} + 3 q^{16} - 3 q^{17} \\&\qquad- q^{18} - 4 q^{19} + O(q^{20}), \\
    X &= -q^{2} + q^{3} + 3 q^{7} + 4 q^{8} - q^{9} + 2 q^{10} - 6 q^{11} - q^{12} - 10 q^{13} + 2 q^{14} - q^{15} - q^{16} \\&\qquad+ 7 q^{17} + 3 q^{18} + 2 q^{19} + O(q^{20}), \\
    Y &= -q^{3} + q^{4} + q^{5} - q^{6} + q^{7} + 4 q^{9} - 4 q^{11} - q^{12} - 5 q^{13} - q^{15} - 3 q^{16} + 5 q^{17} \\&\qquad+ 4 q^{18} + 3 q^{19} + O(q^{20}), \\
    Z &= q^{6} - q^{7} - q^{8} - q^{9} - q^{10} + 3 q^{11} + 4 q^{13} + q^{16} - 4 q^{17} - 3 q^{18} - 2 q^{19} + O(q^{20})
   \end{align*}} \\ & \\
\hline
\end{tabular}
\end{center}

\begin{center}
\begin{tabular}{|l|l|}
  \hline
  \multirow{42}{*}{$X_{0}^{+}(199)$} & \multirow{3}{14cm}{$g = 4$, equations for canonical embedding in $\mathbf{P}^{g - 1}$ are \begin{align*}
    & 2 W Y - X^2 + 3 X Y - 6 X Z - 5 W Z + 3 Y Z - 6 Z^2 = 0, \\
    & -W X^2 + 3 W^2 Y + 3 W X Y + 3 W Y^2 + X Y^2 + Y^3 - 7 W^2 Z - 5 W X Z \\
    &\qquad- 8 W Y Z - 2 X Y Z - 3 Y^2 Z + 6 W Z^2 + 4 Y Z^2 - 3 Z^3 = 0,
    \end{align*}}  \\ &  \\  &  \\  &  \\  & \\& \\
   &\multirow{5}{14cm}{with the variables corresponding to the following cusp forms \begin{align*}
    W &= q + q^{3} - 2 q^{4} - 2 q^{5} - q^{6} - 6 q^{7} - 3 q^{9} + q^{10} + 2 q^{11} - q^{12} + 3 q^{13} + 2 q^{14} + 5 q^{16} \\&\qquad- q^{17} + q^{18} - 3 q^{19} + O(q^{20}), \\
    X &= q^{2} + q^{3} - 3 q^{5} - 3 q^{6} - 2 q^{7} - 4 q^{8} - q^{9} + 2 q^{10} + 3 q^{11} + q^{12} + 5 q^{13} + 2 q^{14} + 3 q^{15} + 4 q^{16} \\&\qquad+ 4 q^{17} + q^{18} - 4 q^{19} + O(q^{20}), \\
    Y &= -2 q^{3} + q^{4} + q^{5} + q^{6} + 5 q^{7} + 2 q^{9} - 3 q^{10} - 4 q^{11} + q^{12} - 3 q^{13} - 2 q^{14} + q^{15} - 6 q^{16} \\&\qquad- 3 q^{17} - q^{18} + 5 q^{19} + O(q^{20}), \\
    Z &= -q^{3} + q^{5} + q^{6} + 2 q^{7} + q^{8} + q^{9} - 2 q^{10} - 2 q^{11} - 2 q^{13} - q^{14} - 3 q^{16} - 2 q^{17} - q^{18} + 3 q^{19} + O(q^{20})
   \end{align*}} \\ & \\
\hline
   \end{tabular}
\end{center}

\begin{center}
\begin{tabular}{|l|l|}
  \hline
  \multirow{40}{*}{$X_{0}^{+}(251)$} & \multirow{3}{14cm}{$g = 4$, equations for canonical embedding in $\mathbf{P}^{g - 1}$ are \begin{align*}
    & W X - W Y + X Y - 2 Y^2 + 2 W Z + X Z + 4 Y Z + Z^2 = 0, \\
    & W^2 X - W^2 Y + X^2 Y - W Y^2 - 2 X Y^2 + 2 Y^3 + W^2 Z - W X Z  \\
    &\qquad- 3 X^2 Z + 2 W Y Z + 7 X Y Z - 3 Y^2 Z - 21 X Z^2 + 10 Y Z^2 - 28 Z^3 = 0,
    \end{align*}}  \\ &  \\  &  \\  &  \\  & \\& \\
   &\multirow{5}{14cm}{with the variables corresponding to the following cusp forms \begin{align*}
    W &= -q - q^{2} + 2 q^{4} + 3 q^{5} + q^{7} + 3 q^{8} + 3 q^{9} + 2 q^{10} - q^{11} + q^{13} + q^{14} \\&\qquad+ q^{15} - 3 q^{16} - 3 q^{17} + 2 q^{18} + 4 q^{19} + O(q^{20}), \\
    X &= q^{2} - q^{3} - q^{4} + 2 q^{5} - q^{6} - q^{7} - 2 q^{8} + q^{10} - q^{11} + 2 q^{12} - q^{14} \\&\qquad+ 3 q^{16} - 6 q^{17} + 4 q^{19} + O(q^{20}), \\
    Y &= q^{2} - q^{4} - 2 q^{5} + q^{6} + q^{7} - 2 q^{8} - q^{9} - 2 q^{10} + 2 q^{11} - q^{12} + q^{13} - 2 q^{14} \\&\qquad+ 3 q^{16} + 4 q^{17} - 2 q^{18} - 4 q^{19} + O(q^{20}), \\
    Z &= -q^{5} + q^{6} + q^{7} - q^{10} + q^{11} - q^{12} - q^{14} + 3 q^{17} - q^{18} - 3 q^{19} + O(q^{20})
   \end{align*}} \\ & \\
\hline
   \end{tabular}
 \end{center}

\begin{center}
\begin{tabular}{|l|l|}
  \hline
   \multirow{35}{*}{$X_{0}^{+}(311)$} & \multirow{3}{14cm}{$g = 4$, equations for canonical embedding in $\mathbf{P}^{g - 1}$ are \begin{align*}
    &X^2 + W Y - 2 X Y + 2 Y^2 + 7 X Z - 8 Y Z + 13 Z^2 = 0, \\
    &W X^2 - 2 W X Y + X^2 Y - W Y^2 - X Y^2 - 2 Y^3 + W^2 Z + 6 W X Z  \\
    &\qquad- X^2 Z - W Y Z + 5 X Y Z + 4 Y^2 Z + 7 W Z^2 - 4 X Z^2 - 2 Z^3 = 0,
    \end{align*}}  \\ &  \\  &  \\  &  \\  & \\& \\
   &\multirow{5}{14cm}{with the variables corresponding to the following cusp forms \begin{align*}
    W &= q - 2 q^{4} - q^{5} + q^{6} - 2 q^{9} - q^{10} - q^{11} - 3 q^{13} - q^{14} - 2 q^{15} + 3 q^{16} - q^{18} - q^{19} + O(q^{20}), \\
    X &= q^{2} + 2 q^{3} - 3 q^{5} - q^{6} - 3 q^{7} - 4 q^{8} - 2 q^{9} - 3 q^{12} + q^{13} + q^{14} + q^{15} + q^{16} + 7 q^{17} - q^{18} \\&\qquad+ 2 q^{19} + O(q^{20}), \\
    Y &= -q^{3} + q^{4} + q^{5} - q^{6} + q^{9} + q^{12} + q^{13} + q^{15} - 3 q^{16} - 3 q^{17} + q^{18} - q^{19} + O(q^{20}), \\
    Z &= -q^{3} + q^{5} + q^{7} + q^{8} + q^{9} + q^{12} - q^{14} - q^{16} - 3 q^{17} - q^{19} + O(q^{20})
   \end{align*}} \\ & \\
\hline
\end{tabular}
\end{center}
According to Galbraith \cite{Galbraith_1999}, the $j$-invariant of the $\mathbf{Q}$-curve corresponding to the exceptional point is
\begin{align*}
    j = &\,31244183594433270730990985793058589729152601677824000000 \\
    &\quad \pm 1565810538998051715397339689492195035077551267840000\sqrt{39816211853}.
\end{align*} 
\newpage
\subsection{Genus 5}

\begin{center}
\begin{tabular}{|l|l|}
  \hline
   \multirow{40}{*}{$X_{0}^{+}(157)$} & \multirow{3}{14cm}{$g = 5$, equations for canonical embedding in $\mathbf{P}^{g - 1}$ are \begin{align*}
    &V X - V Y + W Y - X Y + Y Z = 0, \\
    &V^2 - V W + W X + W Y + X Y + V Z - W Z + X Z = 0 \\
    &2 V X + W X - V Y - X Y - V Z - W Z + 2 X Z - Z^2 = 0,
    \end{align*}}  \\ &  \\  &  \\  &  \\  & \\
   & \multirow{5}{14cm}{with the variables corresponding to the following cusp forms \begin{align*}
    V &= -q^{4} + q^{5} + q^{6} + 2 q^{8} - q^{9} - 2 q^{10} - 2 q^{11} - q^{13} + 2 q^{14} - q^{15} - q^{17} - q^{18} + 3 q^{19} + O(q^{20})\\
    W &= -q + q^{2} + q^{3} - q^{4} + q^{5} + q^{7} + 2 q^{8} - q^{10} + 2 q^{11} + q^{13} + q^{15} - q^{16} 
    + 3 q^{17} - q^{18} + q^{19} + O(q^{20}) \\
    X &= -q^{2} + q^{3} + q^{4} + q^{8} - 2 q^{9} + 2 q^{10} + q^{12} - q^{13} + 2 q^{14} - 2 q^{15} 
    - q^{16} - 4 q^{17} + 2 q^{18} + q^{19} + O(q^{20})\\
    Y &= q^{5} - q^{6} - 2 q^{7} + q^{8} + q^{11} + 2 q^{12} + q^{13} + q^{14} - q^{15} - 3 q^{16} 
    - 2 q^{17} + 3 q^{18} - q^{19} + O(q^{20})\\
    Z &= -q^{2} + 2 q^{4} + q^{5} + q^{6} - q^{7} - q^{8} + q^{9} + q^{10} + q^{11} - 2 q^{12} - 2 q^{15} - 
    2 q^{17} - q^{18} - 2 q^{19} + O(q^{20})
   \end{align*}} \\ & \\ & \\
\hline
\end{tabular}
\end{center}

Galbraith has a small typo in his model for $X_{0}^{+}(157)$. His third equation should have leading coefficient 
$2 w^{2}$ instead of $w^{2}$. The following linear map sends our model to his model
\begin{align*}
v &= 3 V - W - X + Y + Z \\
w &= 2 V - 2 X - 2 Y + Z \\
x &= -4 V + 3 X - Y - 3 Z \\
y &= -V + X + Y - Z \\
z &= V - X + Z
\end{align*}
and for reference, the corrected Galbraith equations should be as follows:
\begin{align*}
&-v x + 2 w x - 5 v y - 8 v z + 5 w z - 2 x z + 5 y z + z^2 = 0, \\
&w^2 + w x + 4 w y - x y + y^2 + v z + 4 w z - 2 x z - 9 y z - 12 z^2 = 0, \\
&\mathbf{2} w^2 + 3 w x + x^2 - v y + 9 w y + 2 x y + 6 y^2 + v z + 12 w z + 5 x z + 3 y z - z^2 = 0.
\end{align*}

\begin{center}
\begin{tabular}{|l|l|}
  \hline
   \multirow{50}{*}{$X_{0}^{+}(181)$} & \multirow{3}{14cm}{$g = 5$, equations for canonical embedding in $\mathbf{P}^{g - 1}$ are \begin{align*}
    &X Y + Y^2 - W Z + X Z + 3 Y Z + 4 Z^2 = 0, \\
    &V W + 3 W X + 3 V Y + W Y + 4 X Y - 2 Y^2 + V Z + 16 W Z - 2 X Z + 18 Y Z - 10 Z^2 = 0, \\
    &V X + 3 X^2 + V Y - 6 X Y - 9 Y^2 + 5 V Z + 10 W Z + 16 X Z - 14 Y Z + 14 Z^2 = 0, \\
    &V^2 W + V W^2 - W^3 + 3 V W X + W^2 X - V X^2 - 2 X^3 - V^2 Y + V W Y - 5 W^2 Y + 
    3 V X Y \\
    &\qquad- 2 W X Y + 4 X^2 Y + V Y^2 - 3 W Y^2 + 10 Y^3 - 3 V^2 Z - V W Z + 2 W^2 Z 
    - 7 V X Z - 15 W X Z \\
    &\qquad- 9 X^2 Z + 2 V Y Z - W Y Z - 3 X Y Z - 8 Y^2 Z + V Z^2 + 
    4 W Z^2 - X Z^2 + 4 Y Z^2 - 2 Z^3 = 0\\
    &V^2 W + V W^2 - W^3 + V^2 X + 3 V W X + W^2 X + 3 V X^2 + X^3 + 2 V^2 Y + 3 V W Y \\
    &\qquad- 7 W^2 Y + 8 W X Y - 6 V Y^2 - 7 W Y^2 + X Y^2 - 2 Y^3 + 4 V^2 Z + 14 V W Z - 8 W^2 Z  \\
    &\qquad+ 14 V X Z + 5 X^2 Z - 5 V Y Z + 4 W Y Z - X Y Z + 4 Y^2 Z + 6 V Z^2 - 6 W Z^2 + 3 X Z^2  = 0
    \end{align*}}  \\ &  \\  &  \\  &  \\  & \\& \\ & \\ & \\
\hline
\end{tabular}
\end{center}

\begin{center}
\begin{tabular}{|l|l|}
  \hline
   \multirow{45}{*}{$X_{0}^{+}(227)$} & \multirow{3}{14cm}{$g = 5$, equations for canonical embedding in $\mathbf{P}^{g - 1}$ are \begin{align*}
    &3 V W + 2 W^2 - V X + W X + X^2 + 3 V Y + 11 W Y + X Y + 2 Y^2 - 5 V Z - 7 W Z  \\
    &\qquad- 7 X Z - 7 Y Z + 3 Z^2 = 0, \\
    &V W + 3 W^2 + V X + 2 W X - X^2 - V Y + 4 X Y - 3 Y^2 - V Z + 2 W Z - 8 X Z 
    + 13 Y Z - 12 Z^2 = 0, \\
    &2 V^2 + 3 V W + 5 W^2 + 2 V X + W X - 4 X^2 + 2 V Y - 7 W Y + 2 X Y + 2 Y^2 - 2 V Z \\
    &\qquad+ 12 W Z - 10 X Z - 7 Y Z + 3 Z^2 = 0, \\
    &2 W^3 + 2 V^2 X - V W X - 3 W X^2 - 2 X^3 - 2 V^2 Y - 4 V W Y - W^2 Y - V X Y - 6 W X Y \\
    &\qquad- X^2 Y - V Y^2 - Y^3 + 3 V^2 Z - V W Z - 2 W^2 Z + 2 V X Z - 4 W X Z + 3 V Y Z - 
    6 W Y Z \\
    &\qquad- 3 X Y Z + 2 Y^2 Z - 3 V Z^2 - W Z^2 + 5 X Z^2 - Y Z^2 - 2 Z^3 = 0  \\
    &V^3 + V^2 W + 5 V W^2 + 2 W^3 - V W X - W^2 X - 2 V X^2 - W X^2 + X^3 - 3 V W Y - W^2 Y \\
    &\qquad+ V X Y + W X Y - 2 X^2 Y - V Y^2 - 3 W Y^2 - X Y^2 - Y^3 - V^2 Z + 3 V W Z + 
    8 W^2 Z \\
    &\qquad- V X Z - 4 W X Z + 5 X^2 Z + V Y Z + 6 X Y Z + 2 Y^2 Z + 3 W Z^2 - 3 X Z^2
    + 4 Y Z^2 - 3 Z^3  = 0
    \end{align*}}  \\ &  \\  &  \\  &  \\  & \\& \\ & \\ & \\
\hline
\end{tabular}
\end{center}

\begin{center}
\begin{tabular}{|l|l|}
  \hline
   \multirow{35}{*}{$X_{0}^{+}(263)$} & \multirow{3}{14cm}{$g = 5$, equations for canonical embedding in $\mathbf{P}^{g - 1}$ are \begin{align*}
    &V W + 2 X^2 + W Y + 2 X Y - 2 X Z - Y Z + Z^2 = 0, \\
    &V X + 2 X^2 + V Y + W Y + X Y - V Z - 2 X Z - Y Z + Z^2 = 0, \\
    &W X - 3 X^2 - V Y - 2 W Y - 4 X Y - V Z + 4 X Z + 3 Y Z - 2 Z^2 = 0, 
    \end{align*}}  \\ &  \\  &  \\  &  \\  & \\ & \\ 
    &\multirow{5}{14cm}{with the variables corresponding to the following cusp forms \begin{align*}
    V &= q^{6} - q^{7} - q^{8} + q^{11} - q^{12} + q^{13} + q^{16} + q^{17} - 2 q^{18} + 2 q^{19} + O(q^{20}), \\
    W &= -q + q^{3} + q^{5} + q^{6} + q^{7} + q^{10} + q^{12} + 3 q^{13} + q^{14} - q^{15} + 2 q^{16} + 2 q^{17} - q^{18} + q^{19} + O(q^{20}), \\
    X &= q^{4} - q^{5} - q^{6} + q^{9} + q^{11} - q^{12} - q^{13} + q^{15} - 3 q^{16} + 2 q^{17} + q^{18} + O(q^{20}), \\
    Y &= q^{3} - q^{4} - q^{7} - 2 q^{9} + q^{10} - q^{11} + q^{12} + q^{13} - q^{15} + 2 q^{16} + 2 q^{17} + q^{19} + O(q^{20}), \\
    Z &= q^{2} - q^{5} - 2 q^{6} - q^{7} - q^{8} - q^{10} + q^{13} - q^{14} + 2 q^{15} - q^{16} + q^{17} + 2 q^{18} + 2 q^{19} + O(q^{20})
   \end{align*}} \\ & \\
\hline
\end{tabular}
\end{center}

\newpage
\subsection{Genus 6}

\begin{center}
\begin{tabular}{|l|l|}
  \hline
   \multirow{43}{*}{$X_{0}^{+}(163)$} & \multirow{3}{14cm}{$g = 6$, equations for canonical embedding in $\mathbf{P}^{g - 1}$ are \begin{align*}
    &V X - U Y - X Y = 0, \\
    &U Y - W Y - V Z = 0, \\
    &U X - W X - U Z - X Z = 0, \\
    &U X + V X + V Y + V Z - Y Z - Z^2 = 0, \\
    &U V + V^2 - V Y + W Y + W Z + X Z = 0, \\
    &U^2 + U V - V W + W X + V Y + X Y = 0
    \end{align*}}  \\ &  \\  &  \\  &  \\  & \\ & \\ & \\ & \\
\hline
\end{tabular}
\end{center}

\begin{center}
\begin{tabular}{|l|l|}
  \hline
   \multirow{40}{*}{$X_{0}^{+}(197)$} & \multirow{3}{14cm}{$g = 6$, equations for canonical embedding in $\mathbf{P}^{g - 1}$ are \begin{align*}
      & W X - U Y + V Y = 0, \\
      & V W + X^2 - U Y - X Y - X Z = 0, \\
      & U X - U Y + V Y + W Y + Y Z = 0, \\
      & U^2 - U V - U W + V W + W^2 + W Z = 0, \\
      & V^2 - V X + V Y + U Z - W Z + X Z - Y Z = 0, \\
      & U^2 - U W + W^2 - X^2 + V Y + W Y + X Z = 0,
    \end{align*}}  \\ &  \\  &  \\  &  \\  & \\ & \\ & \\ & \\
\hline
\end{tabular}
\end{center}

\begin{center}
\begin{tabular}{|l|l|}
  \hline
   \multirow{41}{*}{$X_{0}^{+}(211)$} & \multirow{3}{14cm}{$g = 6$, equations for canonical embedding in $\mathbf{P}^{g - 1}$ are \begin{align*}
      & U V - V^2 + X^2 - U Y + V Y = 0, \\
      & W^2 - V X - W X + W Y - X Y + Y Z = 0, \\
      & U V + V X + W X - U Y - V Z + Y Z = 0, \\
      & V W + U X - V X - W X + X^2 - W Y + Y Z = 0, \\
      & U W - V X - W X + U Y - V Y - W Y + X Z + Y Z = 0, \\
      & U^2 - U W + U X - W X - V Y - V Z + W Z + Y Z = 0,
    \end{align*}}  \\ &  \\  &  \\  &  \\  & \\ & \\ & \\ & \\
\hline
\end{tabular}
\end{center}

\begin{center}
\begin{tabular}{|l|l|}
  \hline
   \multirow{56}{*}{$X_{0}^{+}(223)$} & \multirow{3}{14cm}{$g = 6$, equations for canonical embedding in $\mathbf{P}^{g - 1}$ are \begin{align*}
      & U^2 + V^2 - U W + V X = 0, \\
      & V X + V Y - W Y + X Y + X Z + Y Z = 0, \\
      & V^2 + V W + V Y + W Y - U Z + V Z + W Z - Y Z = 0, \\
      & V^2 + U W - W^2 - U X + W X - V Y - U Z + W Z = 0, \\
      & U V - V W + V X - W X + X^2 - U Y + V Z + X Z = 0, \\
      & U V + V^2 + U W - U X + W X - X^2 + 3 U Y - V Y + U Z + V Z - W Z + Z^2 = 0,
    \end{align*}}  \\ &  \\  &  \\  &  \\  & \\ & \\ & \\ & \\
\hline
\end{tabular}
\end{center}

\begin{center}
\begin{tabular}{|l|l|}
  \hline
   \multirow{50}{*}{$X_{0}^{+}(269)$} & \multirow{3}{14cm}{$g = 6$, equations for canonical embedding in $\mathbf{P}^{g - 1}$ are \begin{align*}
      & U Y - W Z + X Z = 0, \\
      & W Y - Y^2 + V Z - W Z = 0, \\
      & U W + V W - V Y - U Z - V Z - Y Z = 0, \\
      & U V - U W + W^2 - W X - W Y + X Y = 0, \\
      & V^2 - V W - W^2 + W X + U Y + V Y + Y^2 = 0, \\
      & U^2 - V W - W^2 + U X - V X + X^2 - X Z = 0,
    \end{align*}}  \\ &  \\  &  \\  &  \\  & \\ & \\ & \\ & \\
\hline
\end{tabular}
\end{center}

\begin{center}
\begin{tabular}{|l|l|}
  \hline
   \multirow{50}{*}{$X_{0}^{+}(271)$} & \multirow{3}{14cm}{$g = 6$, equations for canonical embedding in $\mathbf{P}^{g - 1}$ are \begin{align*}
      & U V + V X - W X - W Y + V Z = 0, \\
      & U^2 + U V - U W - U Y - X Y - Z^2 = 0, \\
      & U V + V^2 + U X + U Y + X Y + Y^2 + V Z + 2 W Z = 0, \\
      & U^2 - V X - X Y - Y^2 + U Z - V Z + X Z + Y Z = 0, \\
      & V W - V X + U Y + V Y + X Y + Y^2 + U Z + X Z + Y Z = 0, \\
      & U W - V W - U X - V X - U Y - 2 V Y - Y^2 - V Z - W Z + X Z = 0,
    \end{align*}}  \\ &  \\  &  \\  &  \\  & \\ & \\ & \\ & \\
\hline
\end{tabular}
\end{center}

\begin{center}
\begin{tabular}{|l|l|}
  \hline
   \multirow{55}{*}{$X_{0}^{+}(359)$} & \multirow{3}{14cm}{$g = 6$, equations for canonical embedding in $\mathbf{P}^{g - 1}$ are \begin{align*}
      & U V + V X + V Y + U Z + V Z + W Z = 0, \\
      & U V + 2 U W + V W + W^2 + W Y + W Z + X Z = 0, \\
      & U^2 - V^2 - V W - U X + U Y + W Y + Y^2 = 0, \\
      & V^2 + V W + U Y + V Y + U Z + V Z - X Z + Y Z = 0, \\
      & U^2 + U W - U X - W X - W Y + U Z + W Z - X Z = 0, \\
      & U V + U W + V W - V X - W X - U Y - W Y - Y^2 - Y Z = 0,
    \end{align*}}  \\ &  \\  &  \\  &  \\  & \\ & \\ & \\ & \\
\hline
\end{tabular}
\end{center}

\printbibliography
\end{document}